\documentclass[12pt]{article}
\usepackage[T1]{fontenc}
\usepackage[utf8]{inputenc}
\usepackage{authblk}
\usepackage[misc]{ifsym}
\usepackage{dsfont}
\usepackage{mathrsfs}
\usepackage{amsmath,amssymb}
\usepackage{amsfonts}

\usepackage{amsthm}
\usepackage{graphicx}
\usepackage{subfigure}
\usepackage{xcolor}

\usepackage{bbm}
\usepackage{mathrsfs}
\usepackage{txfonts}
\usepackage{color}
\usepackage{pict2e}
\usepackage{float}
\usepackage{palatino,epsfig,latexsym}
\usepackage{epsf,xy,epic,amscd}
\usepackage{lineno}

\usepackage{bibentry}
\usepackage[colorlinks,linkcolor=blue,anchorcolor=blue,citecolor=blue,]{hyperref}

\theoremstyle{plain}
\newtheorem{theorem}{Theorem}[section]
\newtheorem{lemma}[theorem]{Lemma}

\newtheorem{proposition}[theorem]{Proposition}

\newtheorem{Bounded Diameter Lemma}[theorem]{Bounded Diameter Lemma}
\theoremstyle{definition}

\newcommand{\Hmm}[1]{\leavevmode{\marginpar{\tiny%
			$\hbox to 0mm{\hspace*{-0.5mm}$\leftarrow$\hss}%
			\vcenter{\vrule depth 0.1mm height 0.1mm width \the\marginparwidth}%
			\hbox to
			0mm{\hss$\rightarrow$\hspace*{-0.5mm}}$\\\relax\raggedright #1}}}

\hfuzz=\maxdimen
\DeclareFixedFont{\Acknowledgment}{OT1}{cmr}{bx}{n}{14pt}
\begin{document}
\numberwithin{equation}{section}
 \title{The Dirichlet problem for degenerate fully nonlinear elliptic equations on Riemannian manifolds}
\author{Ri-Rong Yuan\thanks{School of Mathematics, \,South China University of Technology, \,Guangzhou 510641, \,China \\ \indent \indent Email address:\,yuanrr@scut.edu.cn
}}
\date{}

\maketitle

\begin{abstract}
	
	We derive the existence of $C^{1,1}$-solutions to the Dirichlet problem for degenerate fully nonlinear elliptic equations on Riemannian manifolds under appropriate assumptions. 


\end{abstract}



 \section{Introduction}

 Let $(M,g)$ be a compact Riemannian manifold of dimension $n\geq2$
 with smooth boundary $\partial M$,  $\bar M:=M\cup\partial M$, 
  with the Levi-Civita connection $\nabla$. Let $\chi$ be a smooth symmetric $(0,2)$-tensor. Let
 $f$ be a symmetric function defined in an open symmetric convex cone  $\Gamma\subset\mathbb{R}^n$ containing positive cone, i.e.
 \[\Gamma_n:=\{\lambda\in \mathbb{R}^n: \mbox{ each component } \lambda_i>0\}\subseteq\Gamma\]
and with vertex at the origin  and with  boundary $\partial \Gamma\neq \emptyset$.
 
 In this paper,
 we consider the Dirichlet problem 
 \begin{equation}
 	\label{mainequ}
 	\begin{aligned}
 		\,& f(\lambda(\nabla^2 u+\chi))= \psi   \mbox{ in } M, \,& u=  \varphi  \mbox{ on }\partial M
 	\end{aligned}
 \end{equation}
with degenerate right-hand side
\begin{equation}
	\label{de-RHS}
	\begin{aligned} 
		\inf_M \psi=\sup_{\partial\Gamma}f,
	\end{aligned} 
\end{equation}
where $\lambda(\nabla^2 u+\chi)$ denote the eigenvalues of $\nabla^2 u+\chi$ with respect to $g$, and
$$\sup_{\partial \Gamma}f =\sup_{\lambda_{0}\in \partial \Gamma } \limsup_{\lambda\rightarrow \lambda_{0}}f(\lambda).$$  
 In addition, some standard assumptions are imposed as follows:
 \begin{equation}
 	\label{continuity1}
 	\begin{aligned}
 		f\in C^\infty(\Gamma)\cap C(\overline{\Gamma}), \mbox{ where } \overline{\Gamma} =\Gamma\cup\partial\Gamma,
 	\end{aligned}
 \end{equation}
 \begin{equation}
 	\label{elliptic}
 	\begin{aligned}
 		\,& f_i(\lambda):=\frac{\partial f}{\partial \lambda_{i}}(\lambda)> 0  \mbox{ in } \Gamma,\,&  \forall 1\leq i\leq n,
 	\end{aligned}
 \end{equation}
 \begin{equation}
 	\label{concave}
 	\begin{aligned}
 		f \mbox{ is  concave in } \Gamma.
 	\end{aligned}
 \end{equation}

 For $M=\Omega\subset\mathbb{R}^n$ and $\chi=0$, the Dirichlet problem for degenerate Monge-Amp\`ere equation  has been studied by  Guan  \cite{GuanP1997Duke} (with homogeneous boundary data) and  Guan-Trudinger-Wang \cite{GTW1999}, which was later extended by Ivochkina-Trudinger-Wang  \cite{ITW2004} to more general degenerate Hessian equations. 
 In a series of papers   
 \cite{HL2009CPAM,HL2011JDG,HL2019}, 
 Harvey-Lawson proposed the concepts of strict $F$ and $\widetilde{F}$ convexity assumptions
 on the boundary  then derived with Perron method the existence and uniqueness of continuous
 solutions   
  to the Dirichlet problem 
 which have broader forms than that of \eqref{mainequ}. 
See also the Survey \cite{HL2013Survey}.
 For more progress and open problems on degenerate nonlinear elliptic equations, we also refer the readers to  \cite{GuanP2004ICCM}.
 
%
 This paper is a complement to these works.
We derive $C^{1,1}$ weak solutions to the Dirichlet problem for degenerate equations on Riemannian
	 manifolds subject to 
\begin{equation}
	\label{bdry-assum1}
	\begin{aligned}
		(-\kappa_1,\cdots,-\kappa_{n-1})\in \overline{\Gamma}_\infty \mbox{ in } \partial M.
	\end{aligned}
\end{equation}
Here  $\kappa_1,\cdots,\kappa_{n-1}$ denote the principal curvatures of the boundary, and $\overline{\Gamma}_\infty$ is the closure of  
$$\Gamma_\infty=\left\{\lambda'\in\mathbb{R}^{n-1}: (\lambda',\lambda_n)\in\Gamma\right\}.$$

 First we present some notion. We say $u\in C^2(\bar M)$ is  \textit{admissible} if 
 \begin{equation}
 	\begin{aligned}
 		\lambda(\nabla^2 u+\chi)\in \Gamma \mbox{ in } \bar M. \nonumber
 	\end{aligned}
 \end{equation}
 We say 
 $\underline{u}$ is a subsolution of the Dirichlet problem \eqref{mainequ} if  
 \begin{equation}
 	\label{subsolution-1}
 	\begin{aligned} 
 		f(\lambda(\nabla^2\underline{u}+\chi))\geq\psi
 		\mbox{ in } \bar M,  \quad
 		\underline{u}=\varphi \mbox{ on } \partial M.
 	\end{aligned} 
 \end{equation}
Furthermore, it is called a strictly subsolution 
 if
 \begin{equation}
 	\label{subsolution-2}
 	\begin{aligned}
 		f(\lambda(\nabla^2\underline{u}+\chi))>\psi 
 		\mbox{ in } \bar M, 
 		\quad   \underline{u}=\varphi \mbox{ on } \partial M.
 \end{aligned}\end{equation}

 \begin{theorem}
 	\label{thm-existence}
 	Let $(M,g)$ be a compact Riemannian manifold 
 	of dimension $n\geq2$ 
 	with smooth boundary satisfying \eqref{bdry-assum1}.
 	Suppose \eqref{continuity1}-\eqref{concave}  hold. For the
 	$\varphi\in C^{2,1}(\partial M)$, $\psi\in C^{1,1}(\bar M)$ satisfying
 	\eqref{de-RHS}, we assume that there exists a $C^{2,1}$ admissible strictly subsolution to the Dirichlet problem \eqref{mainequ}.
Then the Dirichlet problem  admits a weak solution $u\in C^{1,1}(\bar M)$ with 
 	$$\lambda(\nabla^2 u+\chi)\in   \overline{\Gamma}  \mbox{ in } \bar M, \quad \Delta u\in L^\infty(\bar M).$$
 	
 \end{theorem} 
 
	
 Condition \eqref{elliptic} ensures \eqref{mainequ} to be elliptic at any admissible solutions,
 while \eqref{concave} implies that the operator $F(A)=f(\lambda(A))$ is concave with respect to $A$ when $\lambda(A)\in\Gamma$. 
 Consequently, according to Evans-Krylov theorem \cite{Evans82,Krylov83} and classical Schauder theory, higher order estimates for  {admissible} solutions follow from 
 \begin{equation}
 	\label{estimate00}
 	|u|_{C^{2}(\bar M)}\leq C.
 \end{equation}
 
 When $\chi=0$ and $M=\Omega$ is a smooth bounded domain in $\mathbb{R}^n$, the estimate \eqref{estimate00} was established
 by Caffarelli-Nirenberg-Spruck \cite{CNS3}  for the Dirichlet problem under assumptions that the principal curvatures $\kappa_1,\cdots,\kappa_{n-1}$ of $\partial\Omega$ satisfy
 \begin{equation}
 	\label{Gamma-1}
 	\begin{aligned}
 		(\kappa_1,\cdots,\kappa_{n-1})\in \Gamma_\infty  \nonumber
 	\end{aligned}
 \end{equation}
 and $f$  satisfies proper assumptions   including 
 a unbounded condition. 
The bounded case  has been further studied by Trudinger \cite{Trudinger95}   on $\Omega\subset\mathbb{R}^n$. 
 On general closed Riemannian manifolds,  the estimate  \eqref{estimate00} for equation \eqref{mainequ} with $\chi=g$ has been obtained by Li \cite{LiYY1990} when the underlying manifold admits nonnegative sectional curvature, by Urbas \cite{Urbas2002} when replacing nonnegative sectional curvature assumption by certain extra assumptions on $f$ including 
 \begin{equation}
 	\label{key2-yuan}
 	\begin{aligned}
 		f_j(\lambda)\geq \delta \sum_{i=1}^n f_i(\lambda) \mbox{ if } \lambda_j\leq 0.
 	\end{aligned}
 \end{equation}
The Dirichlet problem \eqref{mainequ} 
 has been studied by Guan \cite{Guan12a}, where he imposed 
  either one of such two assumptions, beyond the subsolution assumption.
We shall remark  that in \cite{LiYY1990,Urbas2002,Guan12a}, condition \eqref{key2-yuan} 
 and the assumption of nonnegative sectional curvature  
 are only used to prove gradient estimate.
 The author \cite{yuan-PUE-conformal}\renewcommand{\thefootnote}{\fnsymbol{footnote}}\footnote{The paper  \cite{yuan-PUE-conformal} is essentially extracted from 
 	[arXiv:2011.08580] and [arXiv:2101.04947].} proposed partial uniform ellipticity 
 then confirm the key condition \eqref{key2-yuan} when
 \begin{equation}
 	\label{addistruc}
 	\begin{aligned}
 		\lim_{t\rightarrow +\infty}f(t\lambda)>f(\mu) \mbox{ for any } \lambda, \mbox{ }\mu\in \Gamma,
 	\end{aligned}
 \end{equation} 
thereby 
 extending the a priori
 estimates to a large amount of
  Hessian type equations on Riemannian manifolds.
 (See \cite{yuan2022levelset} for more extension). 
However, in the literature cited above,  the right-hand sides of equations are assumed to satisfy a nondegenerate condition 
 \begin{equation}
 	\label{nondegenerate}
 	\inf_{M} \psi >\sup_{\partial \Gamma} f.
 \end{equation}
Consequently, the estimates established there don't apply to degenerate equations.
 
 In order to 
 prove Theorem \ref{thm-existence}, besides the gradient estimate mentioned above, the other key issue is to bound the second order derivatives at the boundary in terms of a constant depending not on $(\delta_{\psi,f})^{-1}$, where
 $$\delta_{\psi,f}=\inf_M\psi-\sup_{\partial\Gamma}f.$$ 
Such a boundary estimate is fairly delicate.
To this end, following some idea from  \cite{yuan2021-2}\renewcommand{\thefootnote}{\fnsymbol{footnote}}\footnote{The paper  \cite{yuan2021-2} is essentially extracted from
	[arXiv:2203.03439] and the first parts of  
	[arXiv:2001.09238; arXiv:2106.14837].}, we achieve such two ingredients simultaneity through deriving a quantitative  boundary estimate. 
It is very different from previous works 
 	 \cite{GuanP1997Duke,GTW1999,ITW2004}.
 
 \begin{theorem}
 	\label{thm-1}
 	Assume that $(f,\Gamma)$ satisfies  \eqref{elliptic}, \eqref{concave} and \eqref{addistruc}.
 	Let $(M,g)$ be a compact Riemannian manifold with smooth boundary subject to \eqref{bdry-assum1}. Let $\varphi\in C^\infty(\partial M)$ and $\psi\in C^\infty(\bar M)$ satisfy \eqref{nondegenerate}. 
 	Suppose in addition that 
 	the Dirichlet problem \eqref{mainequ} admits a $C^2$-admissible subsolution $\underline{u}\in C^2(\bar M)$ satisfying 
 	\eqref{subsolution-1}.
 	Then for any admissible solution $u\in C^3(M)\cap C^2(\bar M)$ to  \eqref{mainequ}, we have 
 	\begin{equation}
 		\label{quantitative-boundary-estimate}
 		\begin{aligned}
 			\sup_{\partial M}\Delta u \leq C(1+\sup_{M}|\nabla u|^2) 
 		\end{aligned}
 	\end{equation}
 	where
 	$C$ is a constant depending not on $(\delta_{\psi,f})^{-1}$.
 	
 \end{theorem}
 
 	

 The paper is organized as follows. In Section \ref{preli} we summarize some lemmas. They are key ingredients in proof of  Theorem \ref{thm-1}. In Sections \ref{sec3} and \ref{sec4} we derive 
the quantitative boundary estimate. Finally, we complete the proof of Theorem \ref{thm-existence} in Section \ref{sec5}.

  \subsection*{Acknowledgements} 
 The author was supported by the National Natural Science of Foundation of China, Grant No. 11801587.

 \section{Preliminaries}
 \label{preli}

 
 The lemma below plays an important role in the proof of Proposition \ref{mix-prop1}.
 \begin{lemma}
 	[{\cite[Lemma 2.2]{GSS14}}]
 	\label{guan2014}
 	Suppose 
 	\eqref{elliptic} and \eqref{concave} hold.
 	Let $K$ be a compact subset of $\Gamma$ and $\beta_0>0$. There is a constant $\varepsilon>0$ such that,
 	for  $\mu\in K$ and $\lambda\in \Gamma$, when $|\nu_{\mu}-\nu_{\lambda}|\geq \beta_0$,
 	\begin{equation}
 		\label{2nd}
 		\begin{aligned}
 			\sum_{i=1}^n f_{i}(\lambda)(\mu_{i}-\lambda_{i})\geq f(\mu)-f(\lambda)+\varepsilon   (1+\sum_{i=1}^n f_{i}(\lambda)).
 		\end{aligned}
 	\end{equation}
 	Here $\nu_{\lambda}=Df(\lambda)/|Df(\lambda)|$ denotes the unit normal vector to the level set, $\{\lambda'\in\Gamma: f(\lambda')=f(\lambda)\}$, passing through $\lambda$,
 	where $Df(\lambda)=(f_1(\lambda),\cdots, f_n(\lambda))$.
 	
 \end{lemma}

 The following lemmas are key ingredients  for quantitative boundary estimate for  pure normal derivatives.
 
 
 \begin{lemma}
 [\cite{yuan2021-2}]\renewcommand{\thefootnote}{\fnsymbol{footnote}}\footnote{The lemma was proposed in
 	[arXiv:2203.03439] that was  reorganized as a part of \cite{yuan2021-2}.}
 	\label{yuan's-quantitative-lemma}
 	Let $A$ be an $n\times n$ Hermitian matrix
 	\begin{equation}\label{matrix3}\left(\begin{matrix}
 			d_1&&  &&a_{1}\\ &d_2&& &a_2\\&&\ddots&&\vdots \\ && &  d_{n-1}& a_{n-1}\\
 			\bar a_1&\bar a_2&\cdots& \bar a_{n-1}& \mathrm{{\bf a}} 
 		\end{matrix}\right)\end{equation}
 	with $d_1,\cdots, d_{n-1}, a_1,\cdots, a_{n-1}$ fixed, and with $\mathrm{{\bf a}}$ variable.
 	Denote the eigenvalues of $A$ by $\lambda=(\lambda_1,\cdots, \lambda_n)$.
 	Let $\epsilon>0$ be a fixed constant.
 	Suppose that  the parameter $\mathrm{{\bf a}}$ in $A$ satisfies  the quadratic
 	growth condition  
 	\begin{equation}
 		\begin{aligned}
 			\label{guanjian1-yuan}
 			\mathrm{{\bf a}}\geq \frac{2n-3}{\epsilon}\sum_{i=1}^{n-1}|a_i|^2 +(n-1)\sum_{i=1}^{n-1} |d_i|+ \frac{(n-2)\epsilon}{2n-3}.
 		\end{aligned}
 	\end{equation}
 	Then the eigenvalues (possibly with a proper permutation) behave like
 	\begin{equation}
 		\begin{aligned}
 			|d_{\alpha}-\lambda_{\alpha}|
 			<   \epsilon, \mbox{ } \forall 1\leq \alpha\leq n-1;\quad
 			0\leq \lambda_{n}-\mathrm{{\bf a}}
 			< (n-1)\epsilon. \nonumber
 		\end{aligned}
 	\end{equation}
 \end{lemma}

 The concavity of $f$ 
 yields that
 \begin{equation}\label{010} \begin{aligned}
 		\mbox{For any } \lambda, \mbox{  } \mu\in\Gamma,  \mbox{  } \sum_{i=1}^n f_i(\lambda)\mu_i\geq  \limsup_{t\rightarrow+\infty} f(t\mu)/t. \nonumber
 \end{aligned}\end{equation}
 Inspired by this observation,
 the author \cite{yuan2021-2}  introduced the following conditions:
 \begin{equation}
 	\label{addistruc-0}
 	\begin{aligned}
 		\mbox{For any } \lambda\in\Gamma, \mbox{  } \lim_{t\rightarrow+\infty} f(t\lambda)>-\infty,
 	\end{aligned}
 \end{equation}
 \begin{equation}
 	\label{addistruc-1}
 	\begin{aligned}
 		\mbox{For any } \lambda\in\Gamma, \mbox{  } 
 		\limsup_{t\rightarrow+\infty} f(t\lambda)/t\geq0.
 	\end{aligned}
 \end{equation}
 Obviously, it leads to
 \begin{lemma}[\cite{yuan2021-2}]
 	\label{lemma-new-1}
 	Suppose $f$ satisfies \eqref{concave} and \eqref{addistruc-1}.
 	Then  
 	\begin{equation}
 		\label{addistruc-2}
 		\begin{aligned}
 			\sum_{i=1}^n f_i(\lambda)\mu_i\geq 0  \mbox{ for any } \lambda, \mbox{  } \mu\in\Gamma.
 		\end{aligned}
 	\end{equation}
 	In particular
 	\begin{equation}
 		\label{addistruc-3}
 		\begin{aligned}
 			\sum_{i=1}^n f_i(\lambda)\lambda_i\geq0, \quad \forall \lambda\in\Gamma.
 		\end{aligned}
 	\end{equation}
 	If, in addition, $\sum_{i=1}^n f_i(\lambda)>0$
 	then
 	\begin{equation}
 		\label{addistruc-4}
 		\begin{aligned}
 			\sum_{i=1}^n f_i(\lambda)\mu_i>0 \mbox{ for any } \lambda, \mbox{  } \mu\in \Gamma.
 		\end{aligned}
 	\end{equation}
 \end{lemma}
 
 

 We now give characterizations of concave functions  satisfying \eqref{addistruc}.
 \begin{lemma}
 	[\cite{yuan2021-2}]
 	\label{lemma3.4}
 	In the presence of \eqref{elliptic} and \eqref{concave}, the following statements are equivalent to each other.
 	\begin{enumerate}
 		\item[{\bf(1)}] $f$ satisfies  \eqref{addistruc}.
 		\item[{\bf(2)}] $f$ satisfies  \eqref{addistruc-0}.
 		\item[{\bf(3)}] $f$ satisfies  \eqref{addistruc-1}.
 		\item[{\bf(4)}] $f$ satisfies \eqref{addistruc-2}.
 		\item[{\bf(5)}] $f$ satisfies  \eqref{addistruc-3}.
 		\item[{\bf(6)}] $f$ satisfies \eqref{addistruc-4}.
 	\end{enumerate}
 \end{lemma}

 
 

 
 \section{Quantitative boundary estimate for mixed derivatives}
 \label{sec3}

 For a point $x_0\in \partial M$, 
 we shall choose local coordinates
 \begin{equation}
 	\label{coordinate1}
 	\begin{aligned}
 		x=(x_1,\cdots,x_n)
 	\end{aligned}
 \end{equation}
 with origin at $x_0$ 
 such that, when restricted to $\partial M$, $\frac{\partial}{\partial x_n}$ is normal to $\partial M$; 
 moreover, we assume 
 $g_{ij}(x_0)=\delta_{ij}$.

 We will carry out the computations in such local coordinates, and set
 \[\nabla_i=\nabla_{\frac{\partial}{\partial x_i}},
 \quad 
 \nabla_{ij}= 
 \frac{\partial^2}{\partial x_i\partial x_j} -\Gamma_{ij}^k \frac{\partial}{\partial x_k},\]  with a similar convention for higher derivatives, 
 where $\Gamma^k_{ij}$ are the Christoffel symbols   $$\nabla_{\frac{\partial}{\partial x_i}}\frac{\partial}{\partial x_j}=\Gamma_{ij}^k \frac{\partial}{\partial x_k}.$$
 \begin{proposition}
 	\label{mix-prop1}
 	Suppose, in addition to   \eqref{elliptic}, \eqref{concave}  and \eqref{addistruc},  
 	that the data $\varphi\in C^3(\partial M)$,  $\psi\in C^1(\bar M)$ satisfies \eqref{nondegenerate}, and that there is an admissible subsolution $\underline{u}\in C^2(\bar M)$. Then for any $x_0\in\partial  M$, under local coordinate system \eqref{coordinate1}, the admissible solution $u$ of Dirichlet problem \eqref{mainequ} must satisfy 
 	\begin{equation}
 		\begin{aligned}
 			|\nabla_{\alpha n}u(x_0)|\leq C(1+\sup_M|\nabla u|),  \quad \forall 1\leq\alpha\leq n-1
 		\end{aligned}
 	\end{equation}
 	where $C$ is a uniform positive constant depending not on $(\delta_{\psi,f})^{-1}$.
 \end{proposition}
 
 \subsection{Useful formulas and notation}
 Throughout this paper we
 use the notation  \[\mathfrak{g}=\nabla^2 u+\chi, \quad \mathfrak{\underline{g}}=\nabla^2\underline{u}+\chi.\]
 By direct computation one has
 \begin{equation}
 	\label{direct-compute1}
 	\begin{aligned}
 		\nabla_{ij}\nabla_k u    = \frac{\partial^3 u}{\partial x_i\partial x_j \partial x_k}-\Gamma_{ij}^l\frac{\partial^2 u}{\partial x_k\partial x_l}, 
 		 	\end{aligned}
 	\end{equation}
  \begin{equation}
 	\label{direct-compute2}
 	\begin{aligned}
 		\nabla_{ijk}u=\,&
 		\frac{\partial^3 u}{\partial x_i\partial x_j \partial x_k}
 		-\frac{\partial \Gamma_{ij}^l}{\partial x_k}\frac{\partial u}{\partial x_l}   
 		-\Gamma_{ij}^l\Gamma_{kl}^p \frac{\partial u}{\partial x_p}    
 		-\Gamma_{ij}^l \nabla_{kl}u-\Gamma_{ki}^l\nabla_{lj}u -\Gamma_{kj}^l\nabla_{il}u.
 	\end{aligned}
 \end{equation}

 Let $\sigma(x)$ be the distance function from $x\in M$ to $\partial M$, $\rho(x)$ be the distance from $x$ to $x_0$. 
 We denote \[\Omega_{\delta}=\{x\in M: \rho(x)<\delta\}. \] 
 The boundary value condition, $u=\underline{u}=\varphi$ on $\partial M$, implies that  
 \begin{equation}\label{ab1}
 	\begin{aligned}
 		\nabla_\alpha(u-\varphi)=
 		\frac{\nabla_\alpha\sigma}{\nabla_n\sigma}\nabla_n(u-\varphi) \mbox{ on } \partial M\cap\overline{\Omega}_\delta,
 	\end{aligned}
 \end{equation}
 \begin{equation}
 	\label{yuanr-1}
 	\begin{aligned}
 		\nabla_{\alpha\beta} u=\nabla_{\alpha\beta} \underline{u}+\frac{\nabla_{\alpha\beta}\sigma}{\nabla_n\sigma}
 		\nabla_n(u-\underline{u}) \mbox{ at } x_0.
 	\end{aligned}
 \end{equation}
 Let $h$ be the solution to 
 \begin{equation}
 	\label{supersolution1}
 	\begin{aligned}
 		\Delta h+\mathrm{tr}_g\chi=0 \mbox{ in } M, 
 		\quad h=\varphi \mbox{ on } \partial M.
 	\end{aligned}
 \end{equation}
 The maximum principle also yields 
 \begin{equation}
 	\begin{aligned}
 		\underline{u}\leq u\leq h  \mbox{ in } \partial M,
 	\end{aligned}
 \end{equation}
 then at $x_0$
 \begin{equation}
 	\begin{aligned}
 		\nabla_n\underline{u}\leq \nabla_n u\leq \nabla_n h.
 	\end{aligned}
 \end{equation}
 In particular, we have  $C^0$-bound and boundary gradient estimate 
 \begin{equation}
 	\label{c0-bdr-c1}
 	\sup_M|u|+\sup_{\partial M} |\nabla u|\leq C.
 \end{equation}
 
 Under local coordinates \eqref{coordinate1}, 
 the tangential operator on  boundary is given by
 \begin{equation}
 	\label{tangential-operator12321-meng}
 	\mathcal{T} 
 	=\pm\left({\frac{\partial}{\partial x_{\alpha}}}-\eta\frac{\partial}{\partial x_n}\right),
 	\mbox{ for } 1\leq \alpha \leq n-1,
 \end{equation}
 where $\eta={\nabla_{\alpha}\sigma}/{\nabla_{n}\sigma}$. The \eqref{ab1} means that
 \begin{equation}
 	\label{ab2}
 	\mathcal{T}(u-\varphi)=0 \mbox{ on } \partial M\cap \overline{\Omega}_\delta.
 \end{equation}

 \subsection{Completion of the proof of Proposition \ref{mix-prop1}} 
 
 Let $\mathcal{L}$ be the linearized operator of  
 equation \eqref{mainequ} at $u$. Locally, it is given by
 \begin{equation}
 	\label{linearoperator2}
 	\begin{aligned}
 		\mathcal{L}v= F^{ij}\nabla_{ij} v \quad
 		\mbox{ for } v\in C^{2}( M),  \nonumber
 	\end{aligned}
 \end{equation}
 where $F^{ij}=\frac{\partial F}{\partial a_{ij}}(\mathfrak{g})$.
 First, we have the following lemma. 
 \begin{lemma}
 	\label{DN}
 	Let $u\in C^3(M)\cap C^1(\bar M)$ be an admissible solution to equation \eqref{mainequ}. For some small $\delta>0$, we have
 	\begin{equation}
 		\label{pili1}
 		\begin{aligned}
 			\left|\mathcal{L}(\mathcal{T}(u-\varphi))\right|\leq C\left(1+(1+\sup_{ M}|\nabla u|)
 			\sum_{i=1}^n f_i+  \sum_{i=1}^n f_i|\lambda_{i}|\right), \mbox{ in } \Omega_{\delta} 
 		\end{aligned}
 	\end{equation}
 	where  $C$ depends on
 	$|\varphi|_{C^{3}(M)}$, $|\chi|_{C^{1}(M)}$, $\psi_{C^{1}(M)}$
 	and other known data (but not on $(\delta_{\psi,f})^{-1}$).
 	
 \end{lemma}
 
 \begin{proof}
 	Differentiating the equation one has 
 	\begin{equation}
 		\begin{aligned}
 			F^{ij} \nabla_{ij\alpha}u =\nabla_\alpha \psi. \nonumber
 		\end{aligned}
 	\end{equation}
 	Combining with \eqref{direct-compute1}-\eqref{direct-compute2} one derives \eqref{pili1}.
 \end{proof}
 Proposition \ref{mix-prop1} can be proved by constructing barrier functions similar to that used in \cite{yuan2021-2}
 in complex variables. The construction of this type of  barriers follows \cite{Guan1993Boundary}. 
 Let's take
 \begin{equation}
 	\label{ggg}
 	\begin{aligned}
 		\widetilde{\Psi}=\,&
 		A_{1}  \sqrt{b_{1}}(\underline{u}-u)-A_{2}\sqrt{b_{1}}\rho^{2}+A_{3}\sqrt{b_{1}}(N\sigma^{2}-t\sigma)
 		\\ \,&
 		+\frac{1}{\sqrt{b_1}}\sum_{\tau<n}|\nabla_\tau(u-\varphi)|^2+ \mathcal{T}(u-\varphi), \nonumber
 	\end{aligned}
 \end{equation}
 where $b_{1}=1+\sup_M|\nabla u|^{2}$.

 Let $\delta>0$ and $t>0$ be sufficiently small such that $N\delta-t\leq 0$
 (where $N$ is a positive constant sufficiently large to be determined), $\sigma$ is $C^2$  in $\Omega_{\delta}$  
 and
 \begin{equation}
 	\label{bdy1}
 	\begin{aligned}
 		\frac{1}{2} \leq |\nabla \sigma|\leq 2,  \mbox{  }
 		|\mathcal{L}\sigma | \leq   C_2\sum_{i=1}^n f_i,     \mbox{  }
 		|\mathcal{L}\rho^2| \leq C_2\sum_{i=1}^n f_{i}  \mbox{ in } \Omega_{\delta}
 	\end{aligned}
 \end{equation}
 and
 \begin{equation}
 	\label{bdygood1101}
 	\begin{aligned}
 		\mathcal{L}(N\sigma^2-t\sigma)=(2N\sigma-t) \mathcal{L}\sigma+2NF^{ij}\nabla_i\sigma \nabla_j\sigma \mbox{ in } \Omega_{\delta}.
 	\end{aligned}
 \end{equation}
 Furthermore, we  can choose $\delta$ and $t$  small enough  such that $|2N\delta-t|$ is small.

 By straightforward calculation and 
 $|a-b|^2\geq  \frac{1}{2}|a|^2- |b|^2$, one derives
 \begin{equation}
 	\label{bdygood1}
 	\begin{aligned}
 		\mathcal{L} (\sum_{\tau<n}|\nabla_{\tau} (u-\varphi)|^2  )
 		\geq
 		\sum_{\tau<n}F^{ij} \mathfrak{g}_{\tau i} \mathfrak{g}_{\tau  j}
 		-C_1'\sqrt{b_1} \sum_{i=1}^n  f_{i}|\lambda_{i}| 
 		-C_1' b_1\sum_{i=1}^n  f_{i}
 		-C_1' \sqrt{b_{1}}.
 	\end{aligned}
 \end{equation}
 By \cite[Proposition 2.19]{Guan12a}, there is an index $r$ such that
 \begin{equation}
 	\label{beeee}
 	\begin{aligned}
 		\sum_{\tau<n} F^{ij}\mathfrak{g}_{\tau i}\mathfrak{g}_{\tau j}\geq 
 		\frac{1}{2}\sum_{i\neq r} f_{i}\lambda_{i}^{2}.\\ \nonumber
 	\end{aligned}
 \end{equation}
 
 Using 
 Lemma \ref{lemma3.4} and $\sum_{i=1}^n f_i (\underline{\lambda}_i-\lambda_i)\geq 0$ respectively, we obtain
 \[\sum_{i=1}^n f_i |\lambda_i|=2\sum_{\lambda_i\geq 0}f_i \lambda_i-\sum_{i=1}^n f_i\lambda_i<2\sum_{\lambda_i\geq 0} f_i\lambda_i,\]
 \[\sum_{i=1}^n f_i |\lambda_i|= \sum_{i=1}^nf_i \lambda_i-
 2\sum_{\lambda_i<0}  f_i\lambda_i<\sum_{i=1}^n f_i\underline{\lambda}_i-2\sum_{\lambda_i<0} f_i\lambda_i.\]
In conclusion, we have
 \begin{equation}
 	\label{yuan316}
 	\begin{aligned} 
 		\sum_{i=1}^n f_i |\lambda_i| \leq \frac{\epsilon}{4 \sqrt{b_1}} \sum_{i\neq r} f_i\lambda_i^2+  \frac{ C\sqrt{b_1}}{\epsilon} \sum_{i=1}^n f_i.	
 	\end{aligned}
 \end{equation}
 
 By \eqref{bdy1},  \eqref{bdygood1101}, \eqref{bdygood1}, Lemma \ref{DN} and \eqref{yuan316}, we obtain
 \begin{equation}
 	\label{bdycrucial}
 	\begin{aligned}
 		\mathcal{L}(\widetilde{\Psi}) \geq \,& A_1 \sqrt{b_1} \mathcal{L}(\underline{u}-u)
 		+\frac{1}{2\sqrt{b_1}} \sum_{i\neq r} f_{i}\lambda_{i}^{2}
 		+A_3 \sqrt{b_1}\mathcal{L}(N\sigma^{2}-t\sigma)
 		\\
 		\,&
 		-C_1
 		-C_1 \sum_{i=1}^n  f_{i}|\lambda_{i}|
 		-\left(A_2C_2 +C_1 \right) \sqrt{b_1} \sum_{i=1}^n  f_{i} \\ 
 		\geq \,& A_1 \sqrt{b_1} \mathcal{L}(\underline{u}-u)
 		+2NA_3 \sqrt{b_1} F^{ij}\nabla_i\sigma \nabla_j\sigma
 		\\
 		\,&
 		+A_3  (2N\sigma-t) \sqrt{b_1} \mathcal{L}\sigma
 		-C \sqrt{b_1} \sum_{i=1}^n  f_{i}-C.
 	\end{aligned}
 \end{equation}
 
 
 Next we will prove
 \begin{equation}
 	\label{bdy001good}
 	\begin{aligned}
 		\mathcal{L}\widetilde{\Psi}
 		\geq 0,   \mbox{ in } \Omega_{\delta}  \nonumber
 	\end{aligned}
 \end{equation}
 for $0<\delta\ll1$, if we appropriately choose $A_1\gg A_2\gg A_3>1$, $N\gg1$ and $0<t\ll1$.

 {\bf Case I}: If $|\nu_{\lambda }-\nu_{\underline{\lambda} }|\geq \beta_0$,   then by Lemma \ref{guan2014}
 we have
 \begin{equation}
 	\label{guan-key1}
 	\begin{aligned}
 		\sum_{i=1}^n f_{i}(\underline{\lambda}_i-\lambda_i)  \geq \varepsilon  \left(1+\sum_{i=1}^n f_i\right),
 	\end{aligned}
 \end{equation}
 where we take $\beta_0= \frac{1}{2}\min_{\bar M} dist(\nu_{\underline{\lambda} }, \partial \Gamma_n)$ as above,
 $\varepsilon$ is the positive constant in Lemma \ref{guan2014}. 
 Taking $A_1\gg A_3\gg 1$ we 
 obtain $$\mathcal{L}\widetilde{\Psi} \geq 0 \mbox{ on } \Omega_\delta.$$
 
 {\bf Case II}: Suppose that $|\nu_{\lambda }-\nu_{\underline{\lambda} }|<\beta_0$. Then
 $\nu_{\lambda }-\beta_0 \vec{\bf 1} \in \Gamma_{n}$ and
 \begin{equation}
 	\label{2nd-case1}
 	\begin{aligned}
 		f_{i} \geq  \frac{\beta_0}{\sqrt{n}} \sum_{j=1}^n f_{j}. 
 	\end{aligned}
 \end{equation}
 By \eqref{bdy1}, we have $|\nabla \sigma|\geq\frac{1}{2}$ in $\Omega_\delta$, then 
 \begin{equation}
 	\label{bbvvv}
 	\begin{aligned}
 		2A_1N \sqrt{b_1} F^{i\bar j}\sigma_i \sigma_{\bar j} \geq \frac{A_1N \beta_0 \sqrt{b_1}}{2\sqrt{n}}\sum_{i=1}^n f_i  \mbox{ on } \Omega_\delta.
 	\end{aligned}
 \end{equation}
 On the other hand,    $\mathcal{L}(\underline{u}-u)\geq 0$.
 Thus  $$\mathcal{L}(\widetilde{\Psi}) \geq 0 \mbox{ on }\Omega_\delta, \mbox{ if } A_1N\gg 1.$$

 The boundary value condition $u-\varphi=0$ on $\partial M$ implies  $\mathcal{T}(u-\varphi)=0$ and $|\nabla_\tau(u-\varphi)|\leq C\rho$ on $\partial M\cap \overline{\Omega}_{\delta}$.  Thus
 \begin{equation}
 	\label{yuan-rr1}
 	\begin{aligned}
 		\widetilde{\Psi}= \,&A_{1}\sqrt{b_{1}}(\underline{u}-u)-A_{2}\sqrt{b_{1}}\rho^{2}+A_{3}\sqrt{b_{1}}(N\sigma^{2}-t\sigma)
 		\\ \,&
 		+\frac{1}{\sqrt{b_1}}\sum_{\tau<n}|\nabla_\tau(u-\varphi)|^2+ \mathcal{T}(u-\varphi) 
 		\leq  0, \mbox{ on } \partial   M\cap  \overline{\Omega}_{\delta}.  \nonumber
 	\end{aligned}
 \end{equation}
 Note that $\rho=\delta$ and $\underline{u}-u\leq 0$ on $M\cap \partial\Omega_{\delta}$.
 Hence, if  $A_2\gg 1$ then
 $\widetilde{\Psi}\leq 0$
 on  $M\cap \partial\Omega_{\delta}$,
 where we  use  $N\delta-t\leq 0$.
 Therefore $\widetilde{\Psi}\leq 0$ in $\Omega_{\delta}$ by applying maximum principle. 
 Together with $\widetilde{\Psi}(0)=0$, one has
 $\nabla_\nu\widetilde{\Psi} (0)\leq 0$.  
 Thus
 \begin{equation}
 	\begin{aligned}
 		\nabla_{\nu} \mathcal{T}(u-\varphi)(0) 
 		\leq C' (1+\sup_M |\nabla u|).  \nonumber
 	\end{aligned}
 \end{equation} 
 Here we use \eqref{c0-bdr-c1}.
 Therefore
 \begin{equation}
 	\label{basa}
 	\begin{aligned}
 		\pm\nabla_{\alpha n}u\leq C (1+\sup_{\bar M}|\nabla u|), \mbox{ at } x_0,
 	\end{aligned}
 \end{equation}
 where $C$  depends only on  
 $|\varphi|_{C^{3}(M)}$,
 $|\underline{u}|_{C^{2}(M)}$
 $|\psi|_{C^{1}(M)}$
 and other known data
 (but not on $\sup_{M}|\nabla u|$).
 
 Moreover,   the constant $C$ in \eqref{basa}  does not
 depend on $(\delta_{\psi,f})^{-1}$.

 \section{Quantitative boundary estimate for pure normal derivatives}
 \label{sec4}
 
 Given $x_0\in \partial  M$.
 As in \eqref{coordinate1}, we choose local coordinate $x=(x_1,\cdots,x_n)$    with origin at $x_0$ 
 such that, when restricted to $\partial M$, $\frac{\partial}{\partial x_n}$ is a inner normal vector to $\partial M$, and 
 $g_{ij}(x_0)=\delta_{ij}$.  	Furthermore, we assume $\{\mathfrak{\underline{g}}_{\alpha\beta}\}$ is diagonal at $x_0$.
 
When imposing assumption \eqref{bdry-assum1} on boundary we can derive 
 \begin{proposition}
 	\label{proposition-quar-yuan1}
 	
 	Let $(M,g)$ be a compact Riemannian manifold with smooth boundary satisfying \eqref{bdry-assum1}. 
 	 Suppose, in addition to 
 	 \eqref{elliptic}, \eqref{concave}, \eqref{addistruc} and 
 	 \eqref{nondegenerate},
 	that there exists an admissible subsolution $\underline{u}\in C^2(\bar M)$. Then 
  any admissible solution $u\in C^2(\bar M)$ of the Dirichlet problem \eqref{mainequ} satisfies
 	\begin{equation}
 		\begin{aligned}
 			\nabla_{n n}u(x_0)\leq C\left(1+\sum_{\alpha=1}^{n-1}|\mathfrak{g}_{\alpha n}|^2\right),  \quad \forall x_0\in \partial M,  \nonumber
 		\end{aligned}
 	\end{equation}
 	where $C$ does not depend on $(\delta_{\psi,f})^{-1}$.
 	
 \end{proposition}

 \begin{proof}

 	In what follows the discussion is done at $x_0$, and the Greek letters $\alpha, \beta$ range from $1$ to $n-1$.
 	By \eqref{yuanr-1}
 	\begin{equation}
 		\label{puretangential2}
 		\begin{aligned}
 			\mathfrak{g}_{\alpha\beta} =\underline{\mathfrak{g}}_{\alpha \beta} +\nabla_n(u-\underline{u})\nabla_{\alpha\beta}\sigma
 		\end{aligned}
 	\end{equation}

 	The proof consists of two steps.
 	
 	\noindent{\bf Step 1}.
 	As in 
 	\cite{yuan2021-2} there exist two uniform positive constants $\varepsilon_{0}$, $R_{0}$
 	depending  on  $\mathfrak{\underline{g}}$  and $f$, such that
 	\begin{equation}
 		\label{opppp}
 		\begin{aligned}
 			 \,& 
 			f(\mathfrak{\underline{g}}_{1  1}-\varepsilon_{0}, \cdots, \underline{\mathfrak{g}}_{(n-1) {(n-1)}}-\varepsilon_0, R_0)\geq \psi, \quad
 	 	\\	\,&
 			(\mathfrak{\underline{g}}_{11}-\varepsilon_{0}, \cdots, \underline{\mathfrak{g}}_{(n-1){(n-1)}}-\varepsilon_0, R_0)\in \Gamma.
 		\end{aligned}
 	\end{equation}
 	

 	\noindent{\bf Step 2}.  Next,  we apply 
 	Lemmas \ref{yuan's-quantitative-lemma}  and \ref{lemma3.4} 
 	derive the quantitative boundary estimates for double normal derivative.
 	Let's denote
 	\begin{equation}
 		I(r)=\left(
 		\begin{matrix}
 			1&  &&0\\
 			&\ddots&&\vdots \\
 			& & 1& 0\\
 			0&\cdots& 0& r \nonumber
 		\end{matrix}
 		\right)
 	\end{equation}
 	\begin{equation}
 		A(R)=\left(
 		\begin{matrix}
 			\mathfrak{g}_{1 1} &\mathfrak{g}_{12}
 			&\cdots&\mathfrak{g}_{1 (n-1)}   & \mathfrak{g}_{1 n}\\
 			\mathfrak{g}_{2 1}&\mathfrak{g}_{22}&\cdots& \mathfrak{g}_{2 (n-1)} &\mathfrak{g}_{2n}\\
 			\vdots&\vdots&\ddots& \vdots&\vdots \\
 			\mathfrak{g}_{(n-1) 1}& \mathfrak{g}_{(n-1) 2}& \cdots &  \mathfrak{g}_{{(n-1)}{(n-1)}}& \mathfrak{g}_{(n-1) n}\\
 			\mathfrak{g}_{n 1}&\mathfrak{g}_{n 2}&\cdots& \mathfrak{g}_{n {(n-1)}}& R  \nonumber
 		\end{matrix}
 		\right)
 	\end{equation}
 	\begin{equation}
 		\underline{A}(R)=\left(
 		\begin{matrix}
 			\mathfrak{\underline{g}}_{11}&&  &&\mathfrak{g}_{1 n}\\
 			&\mathfrak{\underline{g}}_{22}&& &\mathfrak{g}_{2 n}\\
 			&&\ddots&&\vdots \\
 			&& &  \mathfrak{\underline{g}}_{{(n-1)} {(n-1)}}& \mathfrak{g}_{(n-1) n}\\
 			\mathfrak{g}_{n1}&\mathfrak{g}_{n2}&\cdots& \mathfrak{g}_{n{(n-1)}}& R  \nonumber
 		\end{matrix}
 		\right)
 	\end{equation}
 	\begin{equation}
 		B=\left(
 		\begin{matrix}
 			\nabla_{1 1}\sigma&	\nabla_{12}\sigma&\cdots&	\nabla_{1 (n-1)}\sigma   & 	0\\
 			\nabla_{2 1}\sigma&	\nabla_{22}\sigma&\cdots&  \nabla_{2 (n-1)}\sigma &0\\
 			\vdots&\vdots&\ddots& \vdots&\vdots \\
 			\nabla_{(n-1) 1}\sigma& \nabla_{(n-1) 2}\sigma& \cdots &  \nabla_{{(n-1)}{(n-1)}\sigma}& 0\\
 			0&0&\cdots& 0& 0  \nonumber
 		\end{matrix}
 		\right)
 	\end{equation}
 	
 	Let $0<\varepsilon_1\ll1$ satisfy 
 	\[\varepsilon_1 \nabla_n(u-\underline{u})<\frac{\varepsilon_0}{8}.\]
 	The assumption \eqref{bdry-assum1} on boundary implies that, for   the fixed $\varepsilon_1\in \mathbb{R}^+$ given above, there exists a uniform positive constant $r_0$ such that
 	\begin{equation}
 		\label{y-43}
 		B+\varepsilon_1 I(r_0/\varepsilon_1)\in\Gamma.	\end{equation}
 	We fix $\varepsilon_0$, $\varepsilon_1$, and $r_0$ that we have chosen. From \eqref{opppp} we see that
 	\begin{equation}
 		\begin{aligned}
 			\lambda(\underline{A}(R)-\varepsilon_1\nabla_n(u-\underline{u})I(r_0/\varepsilon_1))\in \Gamma, \quad R\geq R_0+r_0\nabla_n(u-\underline{u}).
 		\end{aligned}
 	\end{equation}
 	
 	On the other hand, the identity \eqref{puretangential2} means that
 	\begin{equation}
 		\begin{aligned}
 			A(R)= \left[\underline{A}(R)-\varepsilon_1\nabla_n(u-\underline{u})I(r_0/\varepsilon_1)
 			\right]
 			+\nabla_n(u-\underline{u})[B+\varepsilon_1 I(r_0/\varepsilon_1)].  \nonumber
 		\end{aligned}
 	\end{equation}
 	For simplicity, we denote
 	\[\underline{H}(R)=\underline{A}(R)-\varepsilon_1\nabla_n(u-\underline{u})I(r_0/\varepsilon_1).\]	
 Notice that $\nabla_n(u-\underline{u})\geq0$.
 According to Lemma \ref{lemma3.4} we obtain
 	\begin{equation}
 		\label{key-y-1}
 		\begin{aligned}
 			f(A(R))\geq f(\underline{H}(R)).	
 		\end{aligned}
 	\end{equation}

 	Let's pick the parameter $\epsilon=\frac{\varepsilon_0}{8}$  in  Lemma \ref{yuan's-quantitative-lemma},
 	and we assume
 	\begin{equation}
 		\begin{aligned}
 			R_c= \frac{8(2n-3)}{\varepsilon_0}
 			\sum_{\alpha=1}^{n-1} | \mathfrak{g}_{\alpha n}|^2
 			+ (n-1)\sum_{\alpha=1}^{n-1}   \left(|\mathfrak{\underline{g}}_{\alpha \alpha}|+\frac{\varepsilon_0}{8}\right)
 			+\frac{(n-2)\varepsilon_0}{8(2n-3)} +R_0+r_0\nabla_n(u-\underline{u}), \nonumber
 		\end{aligned}
 	\end{equation}
 	where  $\varepsilon_0$ and $R_0$ are the constants from \eqref{opppp}.
 	
 	Lemma \ref{yuan's-quantitative-lemma} applies to
 	$\underline{H}(R_c)$ and
 	the eigenvalues of $\underline{H}(R_c)$ 
 	(possibly with an order) shall behave like
 	\begin{equation}
 		\label{lemma12-yuan}
 		\begin{aligned}
 			\lambda(\underline{H}(R_c)) \in
 			\left(\mathfrak{\underline{g}}_{1\bar 1}-\frac{\varepsilon_0}{4},\cdots,
 			\mathfrak{\underline{g}}_{(n-1) \overline{(n-1)}}-\frac{\varepsilon_0}{4},R_c\right) +\overline{\Gamma}_n \subset \Gamma.
 		\end{aligned}
 	\end{equation}
 	Applying \eqref{elliptic},  
 	\eqref{opppp}, \eqref{key-y-1} and  \eqref{lemma12-yuan},  we hence derive
 	\begin{equation}
 		\begin{aligned}
 			F(A(R_c))\geq  F(\underline{H}(R_c)) 
 			\geq 
 			f(\mathfrak{\underline{g}}_{11}-\frac{\varepsilon_0}{4},\cdots,
 			\mathfrak{\underline{g}}_{(n-1){(n-1)}}-\frac{\varepsilon_0}{4},R_c)
 			> \psi \nonumber
 		\end{aligned}
 	\end{equation}
which then gives $\mathfrak{g}_{n n}(x_0) \leq R_c.$

 \end{proof}

 \section{Proof of existence results}
 \label{sec5}
 

 
 The condition \eqref{addistruc} is necessary for the blow-up argument in \cite{Gabor}. 
 \begin{lemma}
 	[\cite{yuan2021-2}]
 	\label{lemma1-con-addi}
 	Suppose $f$ satisfies \eqref{continuity1}-\eqref{concave}.
 	Then it obeys \eqref{addistruc}.
 \end{lemma}
 
 \begin{proof}
 	Since $f$ satisfies \eqref{continuity1}, we have 
 	\begin{equation}
 		\label{t1-to-0}
 		\begin{aligned}
 			\lim_{t\rightarrow0^+}f(t\vec{\bf1})
 			=f(\vec{\bf 0})>-\infty.
 		\end{aligned}
 	\end{equation}
 For $\lambda\in \Gamma$ there is $t_\lambda>0$ such that $f(t_\lambda\vec{\bf1})=f(\lambda)$.
 	Thus \[f(\lambda)\geq f(\vec{\bf 0}). 
 	\]
 	Consequently,  
 	\eqref{addistruc} holds according to Lemma \ref{lemma3.4}. 
 	
 \end{proof}

 As shown in \eqref{c0-bdr-c1}, we have $C^0$-estimate and boundary gradient estimate
 \[\sup_{M}|u|+\sup_{\partial M}|\nabla u|\leq C.\]
 From Propositions \ref{mix-prop1} and \ref{proposition-quar-yuan1} we deduce Theorem \ref{thm-1}. Namely,
 \begin{equation}
 	\begin{aligned}
 		\sup_{\partial M}\Delta u\leq C (1+\sup_M|\nabla u|^2 ). \nonumber
 	\end{aligned}
 \end{equation}
 By $\Gamma\subset \Gamma_1$, $\Delta u> -\mathrm{tr}_g\chi.$
 Together with the second estimate,  
 \begin{equation}
 	\begin{aligned}
 		\sup_{ M} \Delta u
 		\leq C (1+ \sup_{M}|\nabla u|^{2} +\sup_{\partial M}|\Delta u|) \nonumber
 	\end{aligned}
 \end{equation}
announced in
\cite[Section 8]{Gabor},
 we obtain  
 \begin{equation}
 	\label{quantitative-second-estimate}
 	\begin{aligned}
 		\sup_{M}\Delta u \leq C (1+\sup_{M}|\nabla u|^2).
 	\end{aligned}
 \end{equation}
 With such a second order estimate at hand, we can derive the gradient estimate via a blow-up argument (see \cite{Dinew2017Kolo,Gabor}), thereby 
 establishing \eqref{estimate00}.
 Notice furthermore that the obtained estimates are independent of $(\delta_{\psi,f})^{-1}$,  we can prove Theorem \ref{thm-existence} by standard approximate method.


 



\begin{thebibliography}{99}
 	\medskip
 	
 	\bibitem{CNS3}
 	L.   Caffarelli, L. Nirenberg and J. Spruck, \emph{The Dirichlet
 		problem for nonlinear second-order elliptic equations
 		III: Functions of eigenvalues of the Hessians},
 	{Acta Math.} {\bf 155} (1985),  261--301.
 	
 	
 	
 	
 	\bibitem{Dinew2017Kolo} S. Dinew and  S. Ko{\l}odziej, \textit{Liouville and Calabi-Yau type theorems for complex Hessian equations}, Amer. J. Math. {\bf 139} (2017),  403--415.
 	
 	
 	\bibitem{Evans82} L. C. Evans, \emph{Classical solutions of fully nonlinear convex, second order elliptic equations}, {Comm. Pure Appl. Math.} {\bf 35} (1982),  333--363.
 	
 	
 	
 	\bibitem{Guan12a} B. Guan, \emph{Second order estimates and regularity for fully nonlinear elliptic equations on Riemannian manifolds}, {Duke Math. J.} {\bf 163} (2014), 1491--1524.
 	
 	
 	
 	\bibitem{GSS14} 
 	B. Guan,  S.-J. Shi and  Z.-N. Sui,  \textit{On estimates for fully nonlinear parabolic equations on Riemannian manifolds}, {Anal. PDE.}  {\bf8} (2015), 1145--1164.
 	
 	
 	\bibitem{Guan1993Boundary}
 	B. Guan and J. Spruck, \textit{Boundary-value problems on $\mathbb{S}^{n}$ for surfaces of constant  Gauss curvature}, 
 	{Ann. Math.} {\bf 138} (1993), 601--624.
 	
 	
 	\bibitem{GuanP1997Duke}
 	P.-F. Guan,
 	\emph{$C^{2}$ a priori estimates for degenerate Monge-Amp\`{e}re equations}, Duke Math. J.  {\bf 86} (1997),  323--346.
 	
 	\bibitem{GuanP2004ICCM}
 	P.-F. Guan,
 	Nonlinear degenerate elliptic differential equations. Second International Congress of Chinese Mathematicians, 257-266, New Stud. Adv. Math., 4, Int. Press, Somerville, MA, 2004.
 	
 	\bibitem{GTW1999}
 	P.-F. Guan, N. S. Trudinger and X.-J. Wang, 
 	{\em On the Dirichlet problem for degenerate Monge-Ampère equations}, Acta Math. {\bf 182} (1999),   87--104.
	
	
\bibitem{HL2009CPAM}
F. R. Harvey and H. B. Lawson, Jr.,  
{\em Dirichlet duality and the nonlinear Dirichlet problem}, Comm. Pure Appl. Math. {\bf62} (2009), 
396--443. 
	
\bibitem{HL2011JDG}
F. R. Harvey and H. B. Lawson, Jr.,  
{\em Dirichlet duality and the nonlinear Dirichlet problem on Riemannian manifolds}, J. Differential Geom. {\bf88} (2011),   395--482.

\bibitem{HL2013Survey}
F. R. Harvey and H. B. Lawson, Jr.,
Existence, uniqueness and removable singularities for nonlinear partial differential equations in geometry. 
Surveys in differential geometry. Geometry and topology, 103--156, Surv. Differ. Geom., 18, Int. Press, Somerville, MA, 2013.

\bibitem{HL2019}
F. R. Harvey and H. B. Lawson, Jr.,  
{\em The inhomogeneous Dirichlet problem for natural operators on manifolds},
 Ann. Inst. Fourier (Grenoble) {\bf 69} (2019), 3017--3064.
 	
 	
 	
 	
 	
 	\bibitem{ITW2004}
 	N. M. Ivochkina,  N. S. Trudinger and X.-J. Wang, 
 	{\em The Dirichlet problem for degenerate Hessian equations},
 	Comm. Partial Differential Equations {\bf 29} (2004),   219--235.
 	
 	
   \bibitem{Krylov83} N. V. Krylov,	\emph{Boundedly nonhomogeneous elliptic and parabolic equations in a domain}, {Izvestia Math. Ser.} {\bf 47} (1983), 75--108.
 	
 	\bibitem{LiYY1990} Y.-Y. Li, \emph{Some existence results of fully nonlinear elliptic equations of Monge-Amp\`{e}re type}, {Comm. Pure Appl. Math.}   {\bf43} (1990), 233--271.
 	
 	
 	
 	
 	
 	
 	\bibitem{Gabor} G. Sz\'{e}kelyhidi,  \emph{Fully non-linear elliptic equations on compact Hermitian manifolds}, J. Differential Geom. {\bf 109} (2018),  337--378.
 	
 	
  \bibitem{Trudinger95} N. Trudinger, \emph{On the Dirichlet problem for Hessian equations},  {Acta Math.}  {\bf 175} (1995), 151--164.
 	
 	
 	
 	\bibitem{Urbas2002} J. Urbas,  {Hessian equations on compact Riemannian manifolds}. Nonlinear problems in mathematical physics and related topics, II, vol. 2 of Int. Math. Ser. (N. Y.), Kluwer/Plenum, New York, 2002, pp. 367--377. 
 	
 	
 	
 	
 	
 	
 	
 	
 	
 	
 	
 	
 	
 	
 	\bibitem{yuan2021-2} R.-R. Yuan, \textit{On the regularity  
 		of Dirichlet problem for fully non-linear elliptic equations on Hermitian manifolds}, 
 	arXiv:2203.04898.
 	
 	
 	\bibitem{yuan-PUE-conformal} R.-R. Yuan,
 	\textit{The partial uniform ellipticity and prescribed problems on the conformal classes of complete metrics},
 	arXiv:2203.13212.
 	
 	
 	
 	\bibitem{yuan2022levelset} 	R.-R. Yuan,	\textit{On the level set version of partial uniform ellipticity and applications}, 
 	arXiv:2203.15769.
 \end{thebibliography}
\end{document}